\documentclass[3p,times]{elsarticle}

\usepackage{ecrc}

\volume{215}

\firstpage{1084}

\journalname{Applied Mathematics and Computation}

\runauth{H. Susanto and N. Karjanto}

\jid{amc}

\jnltitlelogo{\normalsize Applied Mathematics and Computation}

\usepackage{amssymb}

\usepackage[figuresright]{rotating}

\newtheorem{theorem}{Theorem}
\newtheorem{proposition}[theorem]{Proposition}
\newenvironment{proof}[1][Proof]{\textbf{#1. }}{\ \rule{0.5em}{0.5em}\\}
\newtheorem{remark}[theorem]{Remark}
\usepackage{amsmath}
\usepackage[all]{nowidow}
\usepackage{subfigure}
\usepackage{url}

\usepackage{fancyhdr}
\begin{document}
\begin{frontmatter}



\dochead{}

\title{Newton's method's basins of attraction revisited}


\author[unup]{H.\ Susanto},
\author[unmc]{N.\ Karjanto\corref{cor1}}
\cortext[cor1]{Corresponding author.}
\ead{natanael.karjanto@nottingham.edu.my}
\address[unup]{School of Mathematical Sciences, University of Nottingham,
University Park, Nottingham, NG7 2RD, UK}
\address[unmc]{Department of Applied Mathematics, Faculty of Engineering, The University of Nottingham Malaysia Campus\\ Semenyih 43500, Selangor, Malaysia}

\begin{abstract} 
In this paper, we revisit the chaotic number of iterations needed by Newton's method to converge to a root. Here, we consider a simple modified Newton method depending on a parameter. It is demonstrated using polynomiography that even in the simple algorithm the presence and the position of the convergent regions, i.e.\ regions where the method converges nicely to a root, can be complicatedly a function of the parameter. 
\hfill \copyright \ 2009 Elsevier Inc. All rights reserved.
\end{abstract}

\begin{keyword}
Newton--Raphson methods \sep iteration methods \sep nodules.
\end{keyword}
\end{frontmatter}


\section{Introduction} \label{intro}



The study of root-finding to an equation through an iterative function has a very long history~\cite{nord22}. 
An algorithm, which is a cornerstone to the modern study of root-finding algorithms was made by Newton through his `method of fluxions'. 
Later on, this method was polished by Rapshon to produce what we now know as the Newton--Raphson method~\cite{nord22,cajo11,ypma95}.

Suppose we want to solve a nonlinear equation $f(z) = 0$ numerically, where $z \in \mathbb{R}$ and the function $f: \mathbb{R} \to \mathbb{R}$ is at least once differentiable. Starting from some $z_0$, the Newton--Raphson method uses the iterations:
\begin{equation}
z_{n + 1} = z_n - \frac{f(z_n)}{f'(z_n)}, \qquad \textmd{for} \; n \in \mathbb{N}. \label{NR1}
\end{equation}
If the initial value $z_0$ is close enough to a simple root of $f(z)$, this iteration converges quadratically. 
If the root is non-simple, the convergence becomes linear. 
Since then, a tremendous amount of effort has been made in the direction of improving the convergency and or the simplicity of the method resulting in modified Newton--Raphson or Newton--Raphson-like methods~\cite{schr70}. For a rather extensive review, see, e.g.,~\cite{orte70,trau82}.

During the last couple of years, the number of papers on Newton--Raphson-like methods is still increasing~\cite{yama00}. 
The work of Abbasbandy~\cite{abba04} on applying Adomian's decomposition method~\cite{adom94} to improve the accuracy of Newton--Raphson method also demonstrates this. 
Some other methods, such as the homotopy method~\cite{he04,liao05}, may also be applied to provide an improved Newton--Raphson method. 
An example is a proposal by Wu~\cite{wu05} where the author employs the method to avoid the singularity in the Newton--Raphson algorithm due to the presence of $f'(z_0) = 0$ for some initial value $z_0$. 
It is clear that the study of iteration methods will still be a versatile field for the coming years. 
At this stage, it might be necessary to have a review note that carefully lists and summarizes most (if not all) of important new proposals and improvements in root-finding algorithms. 
This is needed -- one of which is -- to avoid reinventions of `something known to somebody else'~\cite{petk99,petk07a}. 

As an example, several modified iteration algorithms presented by Sebah and Gourdon~\cite{sg01} have been reinvented recently by Pakdemirli and Boyac{\i}~\cite{pb07}
using perturbation theory.

The Newton--Raphson method~\eqref{NR1} may also fail to converge. 
One of the reasons is if the function $f(z)$ has regions that will cause the algorithm to shoot off to infinity~\cite{kell75}. 
A natural question is then what is the boundary of the points where Newton's method fails to converge. 
The answer is non-trivial particularly if one uses the method to find complex roots $z \in \mathbb{C}$ of complex functions $f: \mathbb{C} \to \mathbb{C}$~\cite{hubb01}. 
The boundary, which is called `Julia set' and its complement `Fatou set', is now known to have a fractal geometry that graphically produces aesthetically pleasing figures~\cite{amat04,varo02,jaco07,peit86}. 
Yet, if one looks at the number of iterations needed for the Newton--Raphson algorithm to converge to a root of the function, the iterative method is chaotic. 
Using starting values $z_0$ close to any of the roots would result in a rapid convergency. 
However, for values where $z_0$ is near the boundaries of the `basins of attraction' in the complex plane, one will find that the method becomes extremely unpredictable~\cite{szys91,pick88,pick90}. 
In this context, a basin of attraction refers to the set of initial conditions leading to long-time behaviour that approaches the attractor(s) of a dynamical system~\cite{ott06}. 
Interestingly, along with the basin boundaries, there are \emph{nodules} that have complex properties~\cite{reev91,reev92}. 
Inside a nodule, the method again becomes regular, i.e.\ the method converges nicely to a root.
\begin{figure}[b!]
\begin{center}
\includegraphics[width=0.45\textwidth]{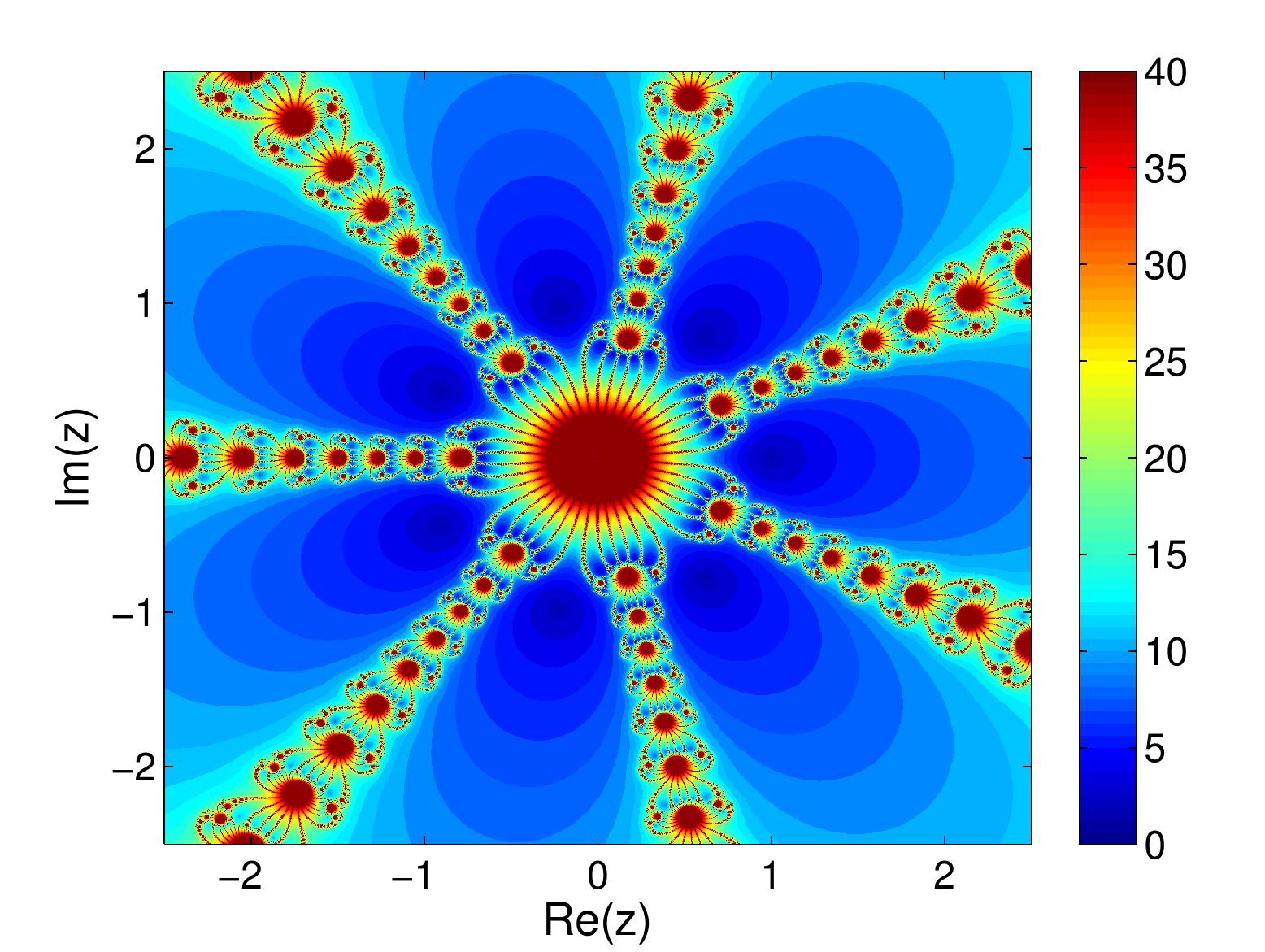}
\includegraphics[width=0.45\textwidth]{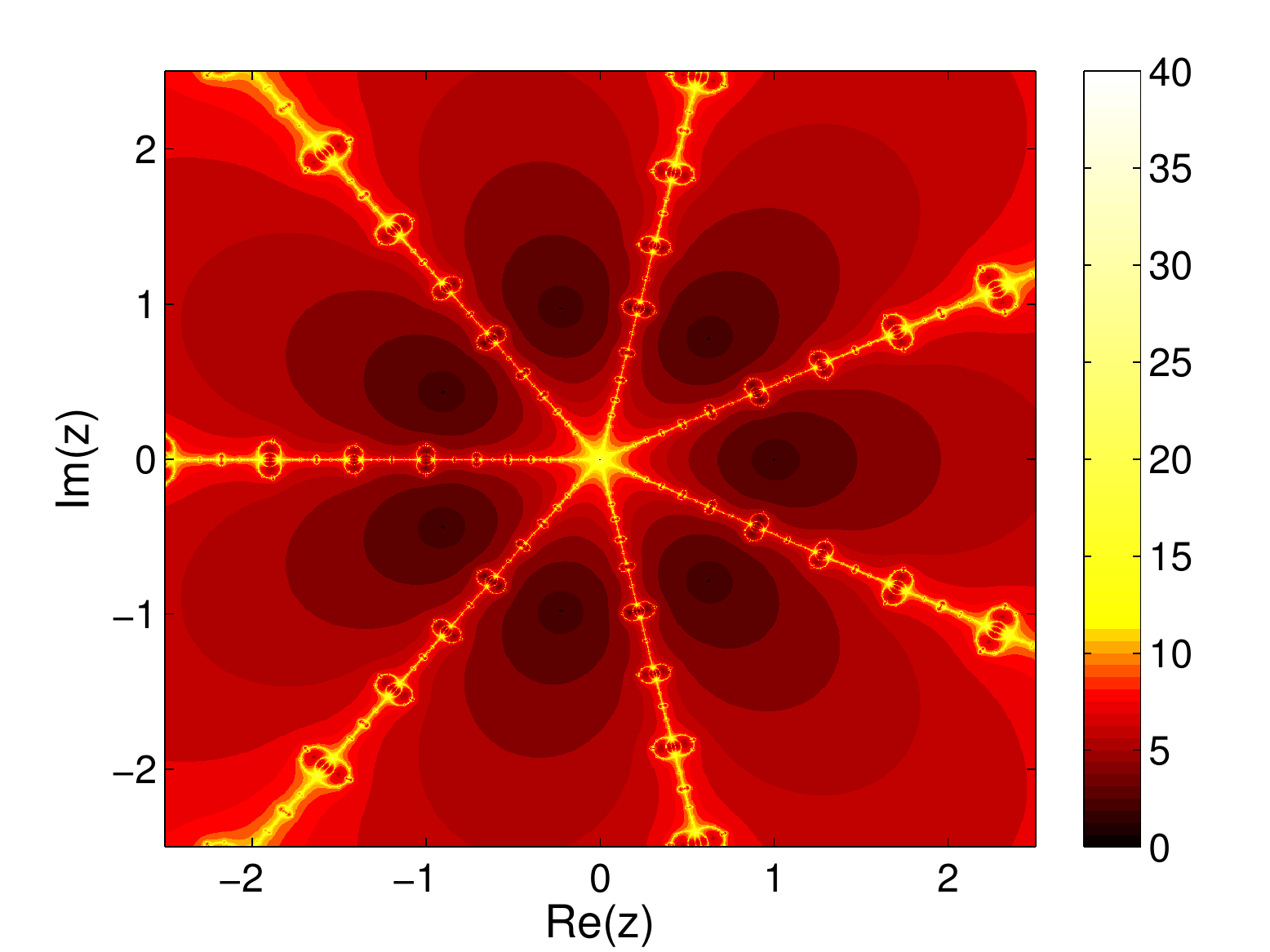}
\includegraphics[width=0.45\textwidth]{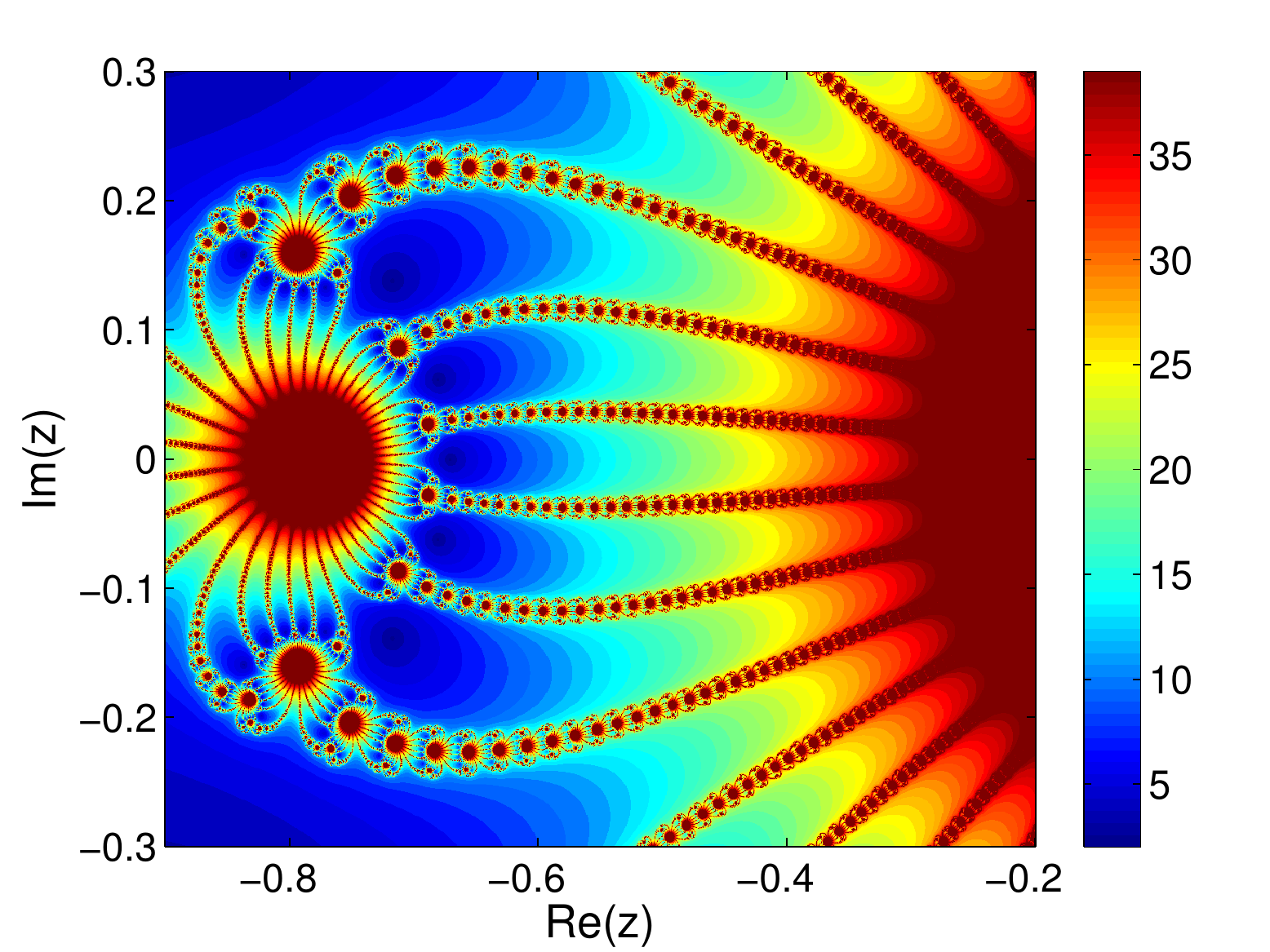}
\includegraphics[width=0.45\textwidth]{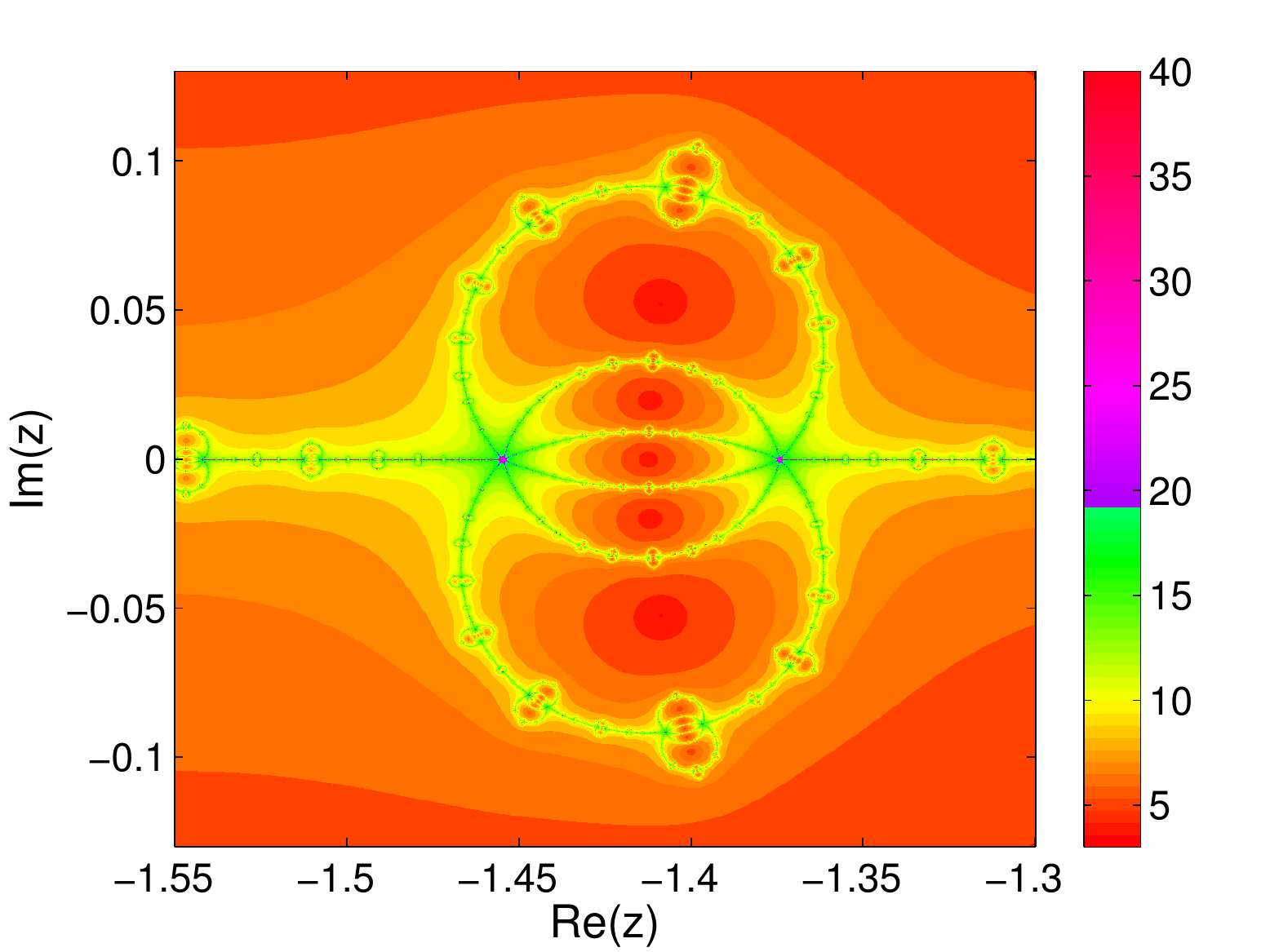}
\end{center}
\caption{A number of iterations needed by Newton's method (right panels) and Halley's method (left panels) to converge to one of the roots of $f(z) = z^7 - 1$ as a function of the initial condition. Bottom panels show one nodule of each method.}
\label{fig1}
\end{figure}

To illustrate, consider the function $f(z) = z^7 - 1$. Using Newton's method, we are looking for the roots of this function. 
Figure~\ref{fig1} shows the number of iterations needed for the algorithm to converge to one root of $f(z)$. 
The iteration is said to be non-convergent if after 40 iterations it does not get close to one of the zeros. 
More technical details in obtaining this figure will be explained in the subsequent sections. 
As a comparison, we also consider the corresponding figure for finding the roots of the function $h(z) = f(z)/\sqrt{f'(z)}$, i.e.\ solving $f(z) = 0$ using Halley's method. 
One can see that non-convergent regions are still present in Halley's method but has a smaller dimension than Newton's method. 
It is then also another purpose of searching modified Newton's methods, i.e.\ to obtain an iterative method that converges to a root with almost any starting point.

Regarding this purpose, to show that a new method is better than an established one, usually the proposer applies his/her new algorithm to find a root of a function and compare its convergency rate with that of an established method. 
Unfortunately, this procedure does not prove anything as the proposer might have randomly picked a starting point that belongs to the convergent regions of the new method. 
This can raise a contradictory result such as shown in, e.g.,~\cite{petk07b}. This is the background of this paper.

In this paper, we revisit the chaotic behaviour of the iteration number needed by Newton's method to converge to a root. 
We consider a family of simple Newton's methods containing a parameter and study the behaviour of the convergent regions as a function of that parameter. 
We will demonstrate that even a simple iterative algorithm can have a complex convergent region.

The analysis presented in this paper employs the techniques of `polynomiography', which is defined as the art and science of visualization in the approximation of the zeros of complex polynomials~\cite{kal04,jin07}. 
Basins of attractions of Newton's methods and of higher order methods were investigated in~\cite{kal04,jin07}. 
However, we investigate here the convergent regions of a modified Newton's method for solving nonlinear equations that depend on two parameters. 
By fixing one parameter and allowing the other parameter to vary, we present some interesting results in connection to the basins of attraction of this modified Newton's methods which are not reported in those two papers.

We outline the present paper as the following. In Section~2, we discuss our modified Newton algorithm. 
Our discussion and analysis of the algorithm are discussed in Section~3. Finally, we conclude and summarize our work in Section~4.

\section{Modified method and convergence analysis}

In the following, we propose to search the roots of a continuously differentiable function $f(z) = 0$ through solving $g(z) = 0$, where $g$ is defined as
\begin{equation}
g(z) = (f(z))^{a_0} \, (f'(z))^{a_1}. \label{gsatu}
\end{equation}
It is obvious that the zero(s) of $f$ would be the zero(s) of $g$ as well. Consequently, the Newton--Raphson iteration~\eqref{NR1} gives
\begin{equation}
z_{n + 1} = F(z_n) = z_n - \frac{f(z_n) \, f'(z_n)}{a_0 \, (f'(z_n))^2 + a_1 \, f(z_n)\,f''(z_n)}. \label{m1}
\end{equation}

The iteration~\eqref{m1} is Newton's method for finding multiple roots when $a_0 = 1, \, a_1 = -1$ and it is Halley's method when $a_0 = 1, \, a_1 = -1/2$. 
When $a_1 = 0$, the iteration~\eqref{m1} is underrelaxed Newton's method for $a_0 > 1$ and overrelaxed one for $a_0 < 1$. 
Hence, we may regard our new function as an interpolant of several methods. 
Using the terminology of Kalantari, we simply consider a modified version of the first member of the `basic family'~\cite{kal00}. 
Then, we have the following convergence proposition. For simplicity, the proposition considers function $f$ applied to real numbers only.
\begin{proposition}
  Let $\alpha \in I$ be a simple zero of a sufficiently differentiable function $f: I \rightarrow \mathbb{R}$ for an open interval $I$. 
  For $a_0 \neq 1$, the iteration method~\eqref{m1} has first order convergence. 
  For $a_0 = 1$ and $a_1 \neq -1/2$, the Newton--Raphson iteration method~\eqref{m1} has second order convergence and satisfies the following error equation:
  \begin{equation}
    e_{n+1} = b_{32} e_{n}^{2} + O(e_{n}^{3})
  \end{equation}
  where $e_{n} = x_{n} - \alpha$ and ${\displaystyle b_{32} = (1/2 + a_{1}) \frac{f''(\alpha)}{f'(\alpha)}}$.
\end{proposition}
\begin{proof}
  Let $\alpha$ be a simple zero of $f$. Consider the iteration function $F$ defined by
  \begin{equation}
      F(x) = x - \frac{f(x) \, f'(x)}{a_0 \, (f'(x))^{2} + a_{1} \, f(x) \, f''(x)}.
  \end{equation}
  Performing some calculations we obtain the following derivatives:
  \begin{eqnarray}
    F(\alpha)   &=& \alpha, \\
    F'(\alpha)  &=& 1 - 1/a_0, \label{c1}\\
    F''(\alpha) &=& \frac{1}{a_0}\left(1 + 2 \frac{a_1}{a_0}\right) \frac{f''(\alpha)}{f'(\alpha)}.
  \end{eqnarray}
  From the Taylor expansion of $F(x_n)$ around $x = \alpha$ we have
  \begin{eqnarray}
    x_{n + 1} &=& F(\alpha) + F'(\alpha)(x_n - \alpha) + \frac{F''(\alpha)}{2!}(x_n - \alpha)^2 + {\mathcal{O}}((x_n - \alpha)^3).
  \end{eqnarray}
  From the first derivative of the iteration function $F$ at $x = \alpha$~\eqref{c1}, it is clear that if $a_0 \neq 1$, then
  \begin{equation}
    x_{n + 1} = \alpha + b_{21} e_n + {\mathcal{O}}(e_n^2), \label{orde1}
  \end{equation}
  where $e_n = x_n - \alpha$ and $b_{21} = 1 - 1/a_0$.
  If one imposes $a_0 = 1$, we have
  \begin{equation}
    x_{n + 1} = \alpha + b_{32} e_n^2 + {\mathcal{O}}(e_n^3), \label{ordedua}
  \end{equation}
  where ${\displaystyle b_{32} = (1/2 + a_{1}) \frac{f''(\alpha)}{f'(\alpha)}}$.
\end{proof}

\begin{remark} \normalfont
When $a_0 = 1$ and $a_1 = -1/2$, the iteration~\eqref{m1} is the well-known Halley's method that has third order convergence as can be proven immediately from equation~\eqref{ordedua}. 
Note that an interesting generalized result connected to the iteration~\eqref{m1} was given by Gerlach~\cite{gerl94}, where it is shown that if
\begin{equation}
  F_{2}(z) = f(z) \qquad \textmd{and} \qquad F_{m + 1}(z) = F_{m}(z) \left[ F_{m}'(z) \right]^{-1/m}, \qquad \textmd{for} \quad m \geq 2,
\end{equation}
then the iterative method
\begin{equation}
  z_{k + 1} = z_{k} - \frac{F_{n + 1}(z_{k})}{F_{n + 1}'(z_{k})} \qquad \textmd{for} \quad k = 0, 1, 2, \dots,
\end{equation}
will have the order of convergence $n + 1$. A particular choice of $n = 2$ also yields Halley's method of the third order convergence.
\end{remark}

\begin{remark} \normalfont
Although one could introduce the function~\eqref{gsatu} in a more generalized form, i.e.
\begin{equation}
  \displaystyle g(z) = \prod_{n=0}^{N} \left(f^{(n)}(z)\right)^{a_n},    \label{gfunc}
\end{equation}
where it is assumed that $f(z)$ is $N$ times continuously differentiable and $a_n \in {\mathbb R}$, $n = 0, 1, 2, \dots, N$, 
it is observed that this generalized function~\eqref{gfunc} does not serve the intended purpose of studying the convergent regions of the modified Newton's method for solving nonlinear equations. 
For $N \geq 2$, it simply does not lead to any higher order methods.
\end{remark}

For the rest of this paper, we will concentrate on the modified Newton's method~\eqref{m1}. 
We fix the parameter $a_0 = 1$ and allow the other parameter $a_1$ to vary. 
In particular, we only consider a function $f(z) = z^7 - 1$ as our chosen example. 
It is observed that the basins of attraction of this convergence method can have a very complicated form.

\section{Numerical results}

Graphical presentations of the number of iterations for the function defined in Equation~\eqref{m1} to converge to a root for two particular values of $a_1$ are shown in Figure~\ref{fig1}.
\begin{figure}[tbh!]
\begin{center}
\includegraphics[width=0.45\textwidth]{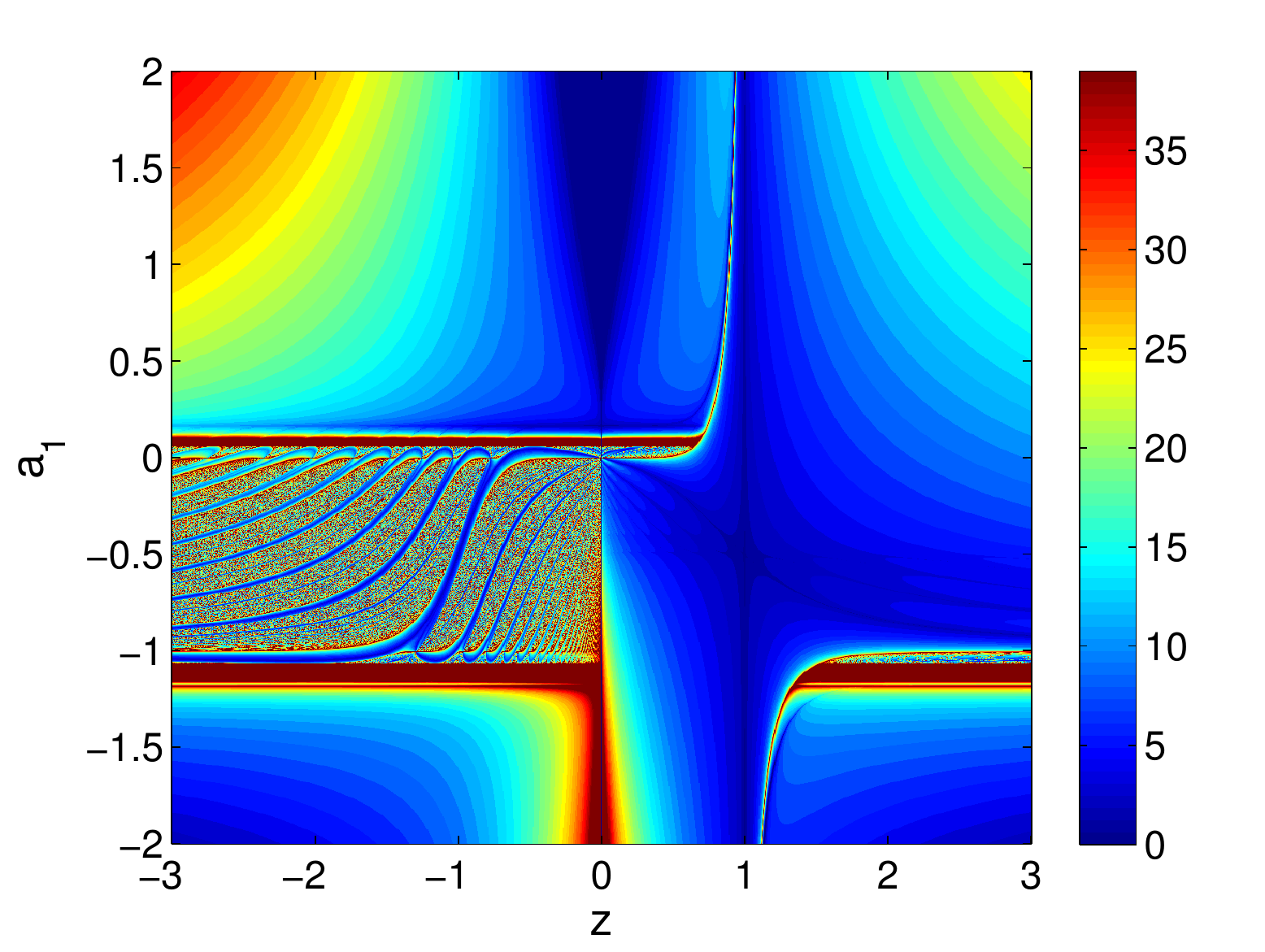}
\includegraphics[width=0.45\textwidth]{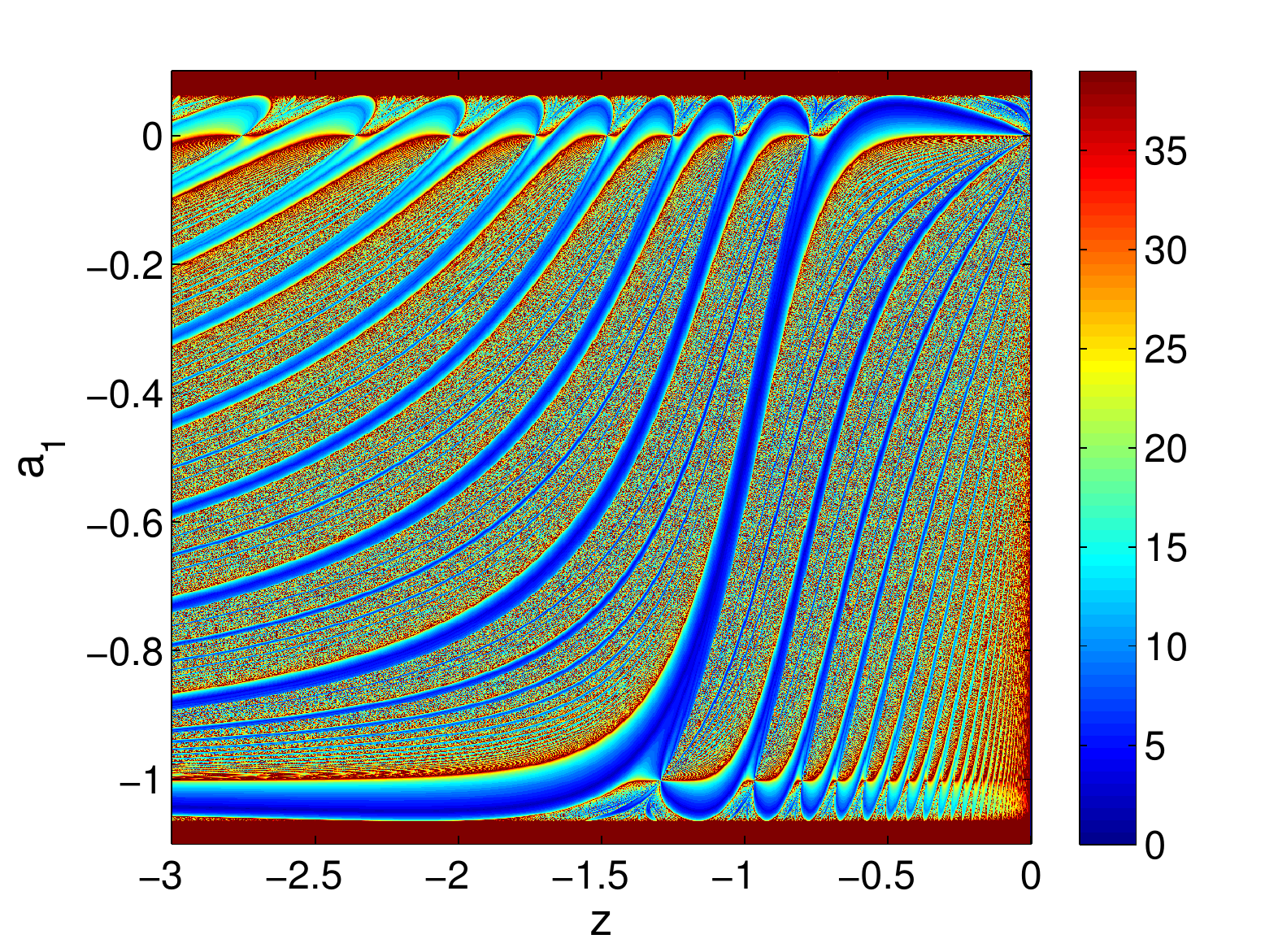}
\includegraphics[width=0.45\textwidth]{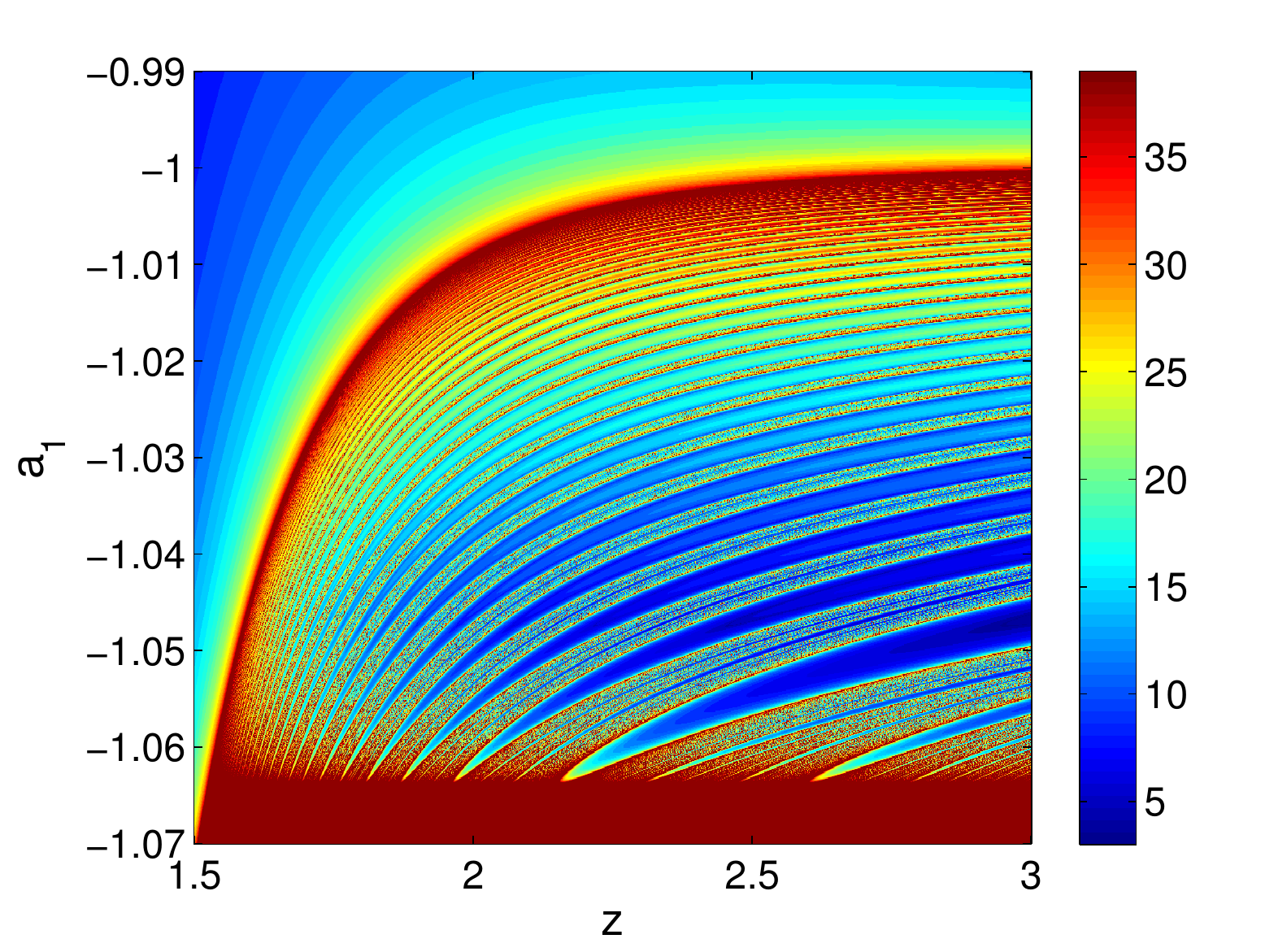}
\includegraphics[width=0.45\textwidth]{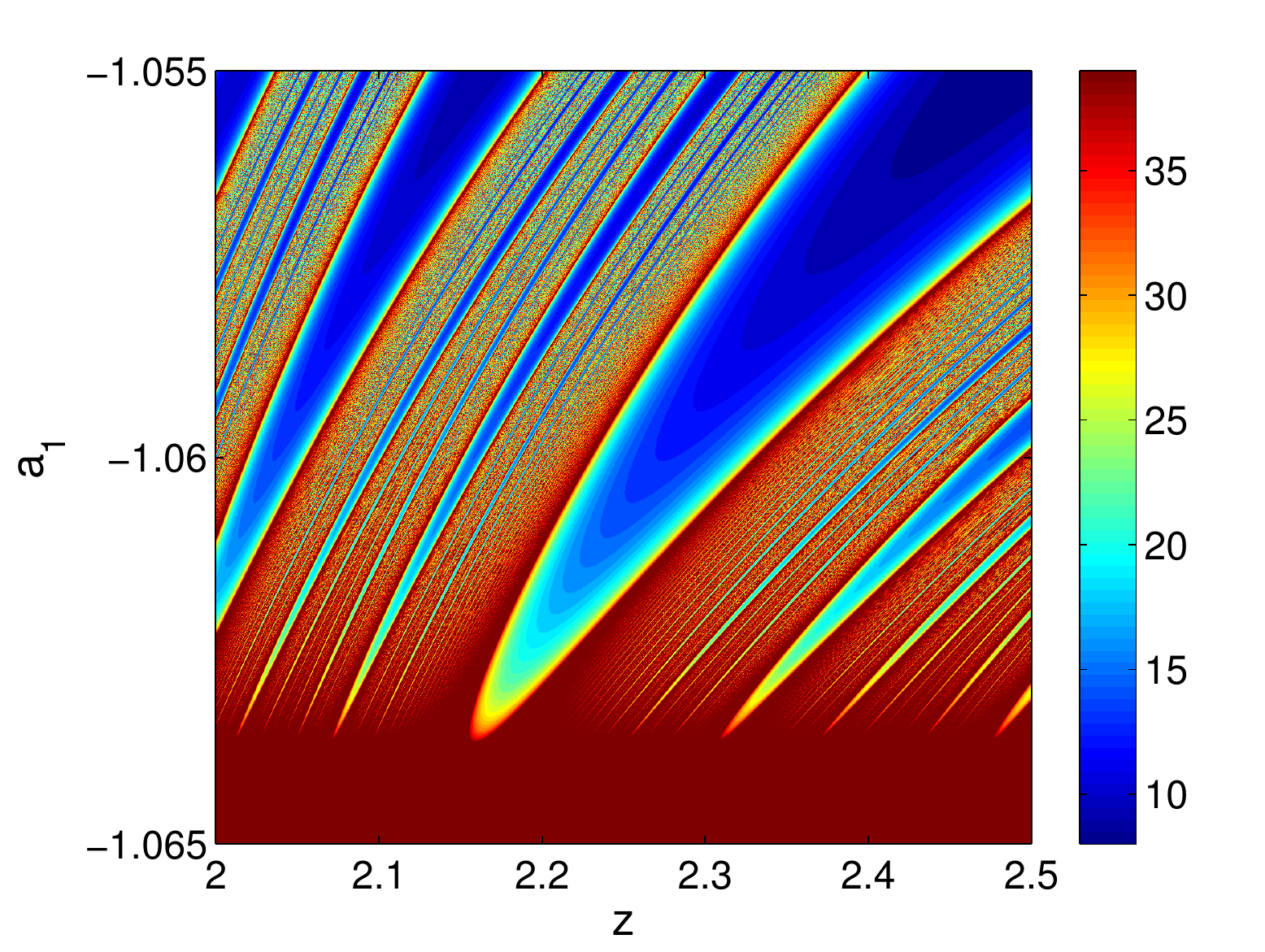}
\end{center}
\caption{Number of iterations needed by iterative method~\eqref{m1} in order to converge to a root of $f(z)$ in the $(z, a_1)$-plane. 
Here, $z \in \mathbb{R}$. Most of the chaotic regions are in the left half plane. 
Yet, there is a critical $a_1$ below which there is a chaotic region in the right half plane.} \label{fig2}
\end{figure}
To obtain the figures, we first take a rectangle $D$ in the complex $z$-plane and then apply the iterative method starting at every point in $D$. 
The numerical method can converge to the root or, eventually, diverges. 
As an illustration, we take a tolerance $\epsilon = 10^{-5}$ and a maximum of 40 iterations. 
So, we take $z_0 \in D$ and iterate $z_{n + 1} = F(z_n)$ up to $|f(z)| < \epsilon$. 
If we have not achieved the desired tolerance within the maximum iteration, then the iterative method is terminated and we say that the initial point does not converge to the root.

We have mentioned that Newton's method has a relatively different structure of basins of attractions as well as nodules from Halley's method. 
It is therefore of interest to see the influence of $\alpha$ on the convergence of the method. 
Since nodules appear along the boundaries of basins of attractions, it is easier if we apply the method in the domain along one of the boundaries. 
For our function $f(z) = z^{7} - 1$, the negative horizontal axis $x < 0$, i.e.\ Im$(z) = 0$, is an example where it separates basins of attractions of root 
$e^{6\pi i/7}$ and $e^{8\pi i/7}$. Our numerical results are presented in Figure~\ref{fig2} where one can observe the presence of regions of chaotic behaviour in the $(z,a_{1})$ plane.
\begin{figure}[tbh!]
\begin{center}
\subfigure[$a_1=0.05$]{\includegraphics[width=0.45\textwidth]{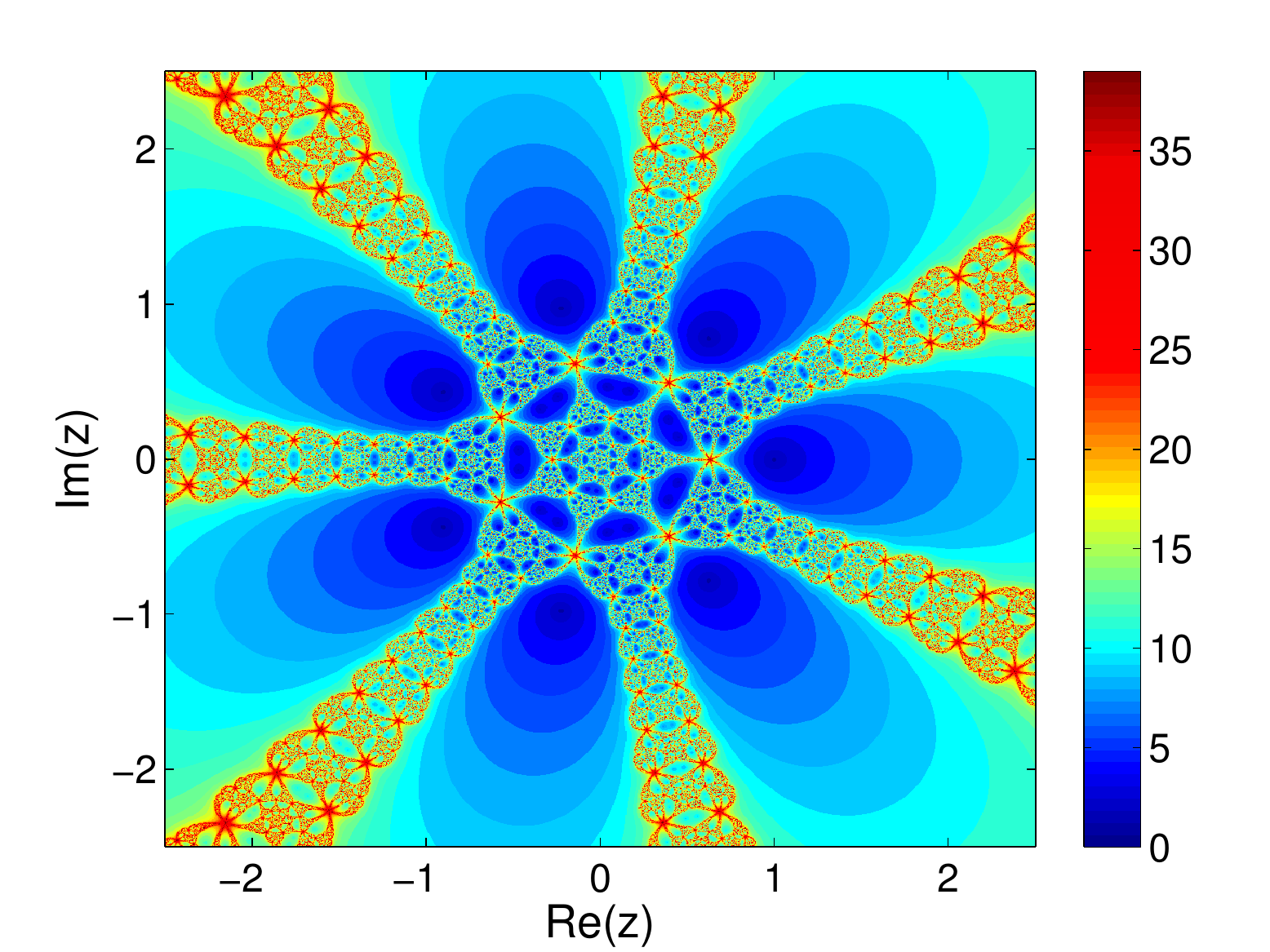}}
\subfigure[$a_1=-0.005$]{\includegraphics[width=0.45\textwidth]{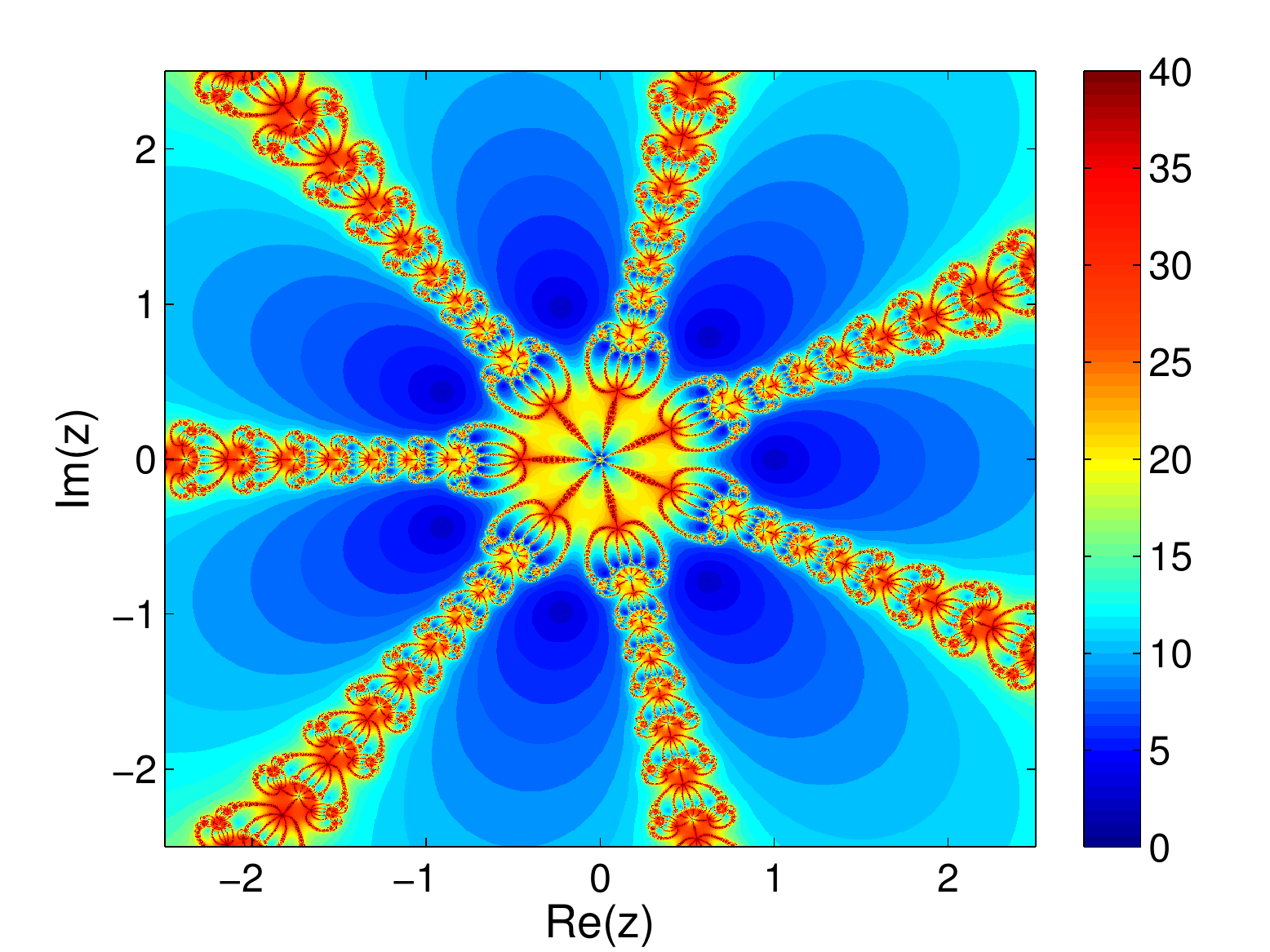}}
\subfigure[$a_1=-0.95$]{\includegraphics[width=0.45\textwidth]{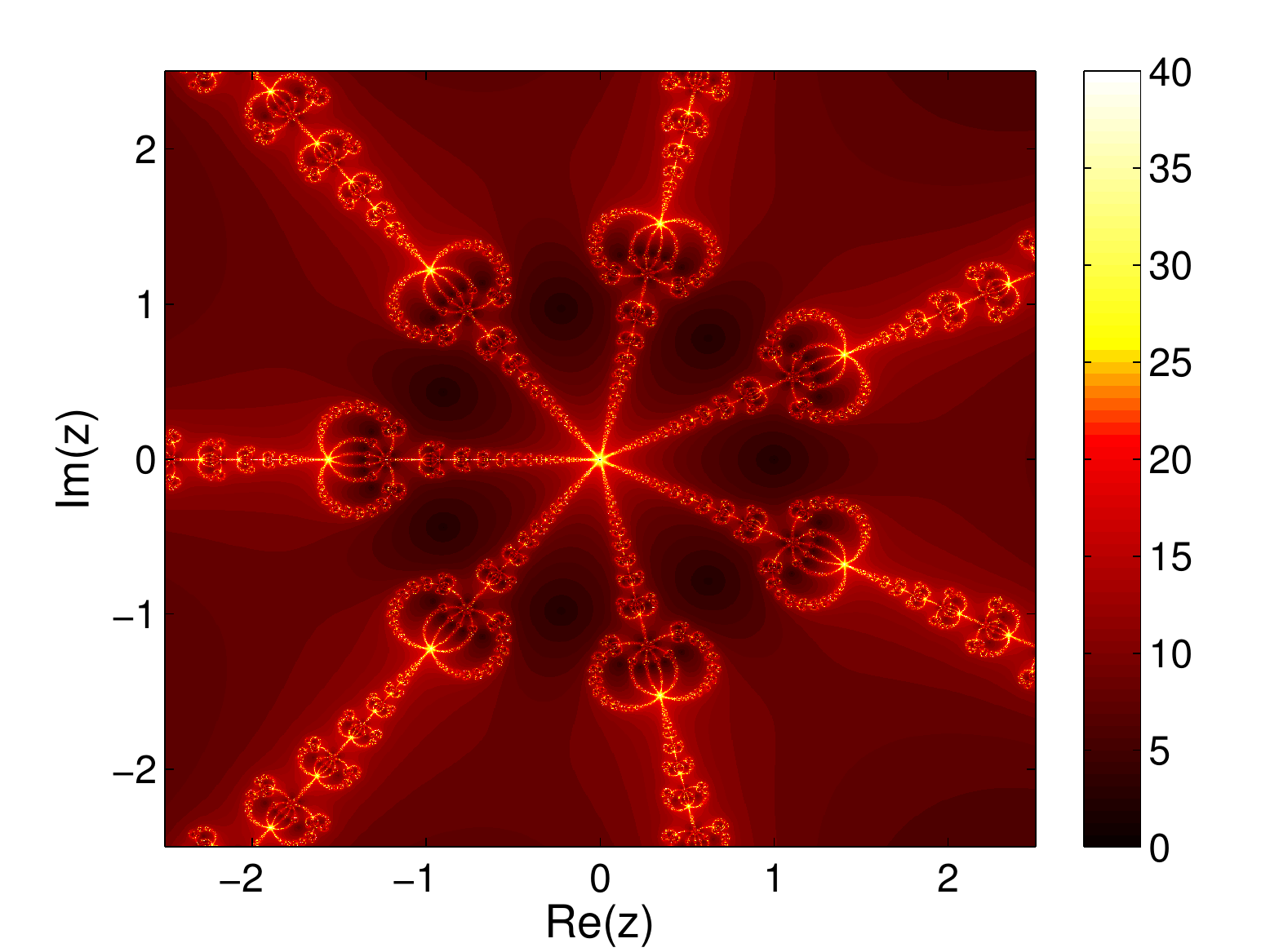}}
\subfigure[$a_1=-1.05$]{\includegraphics[width=0.45\textwidth]{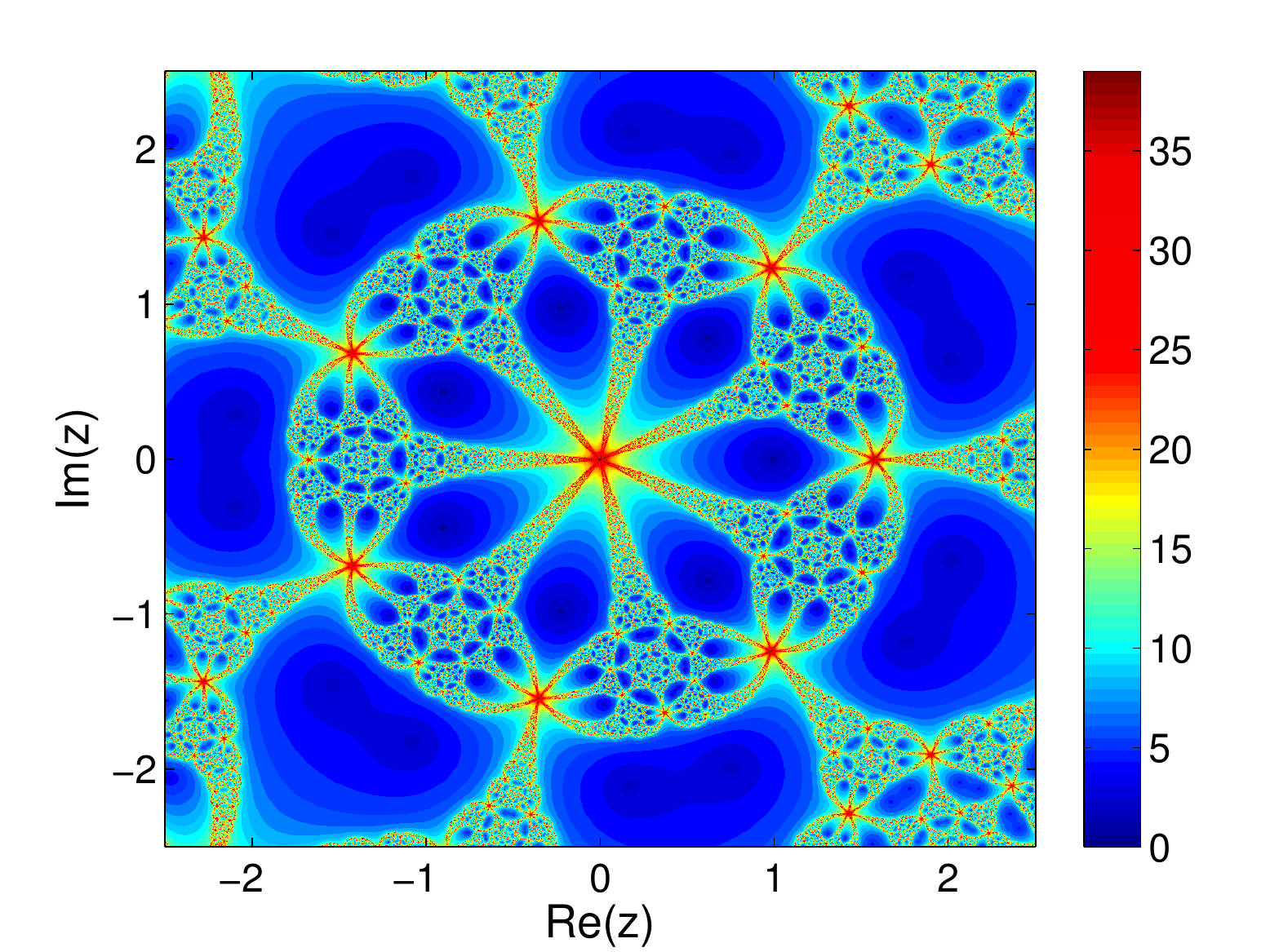}}
\subfigure[$a_1=-1.1$]{\includegraphics[width=0.45\textwidth]{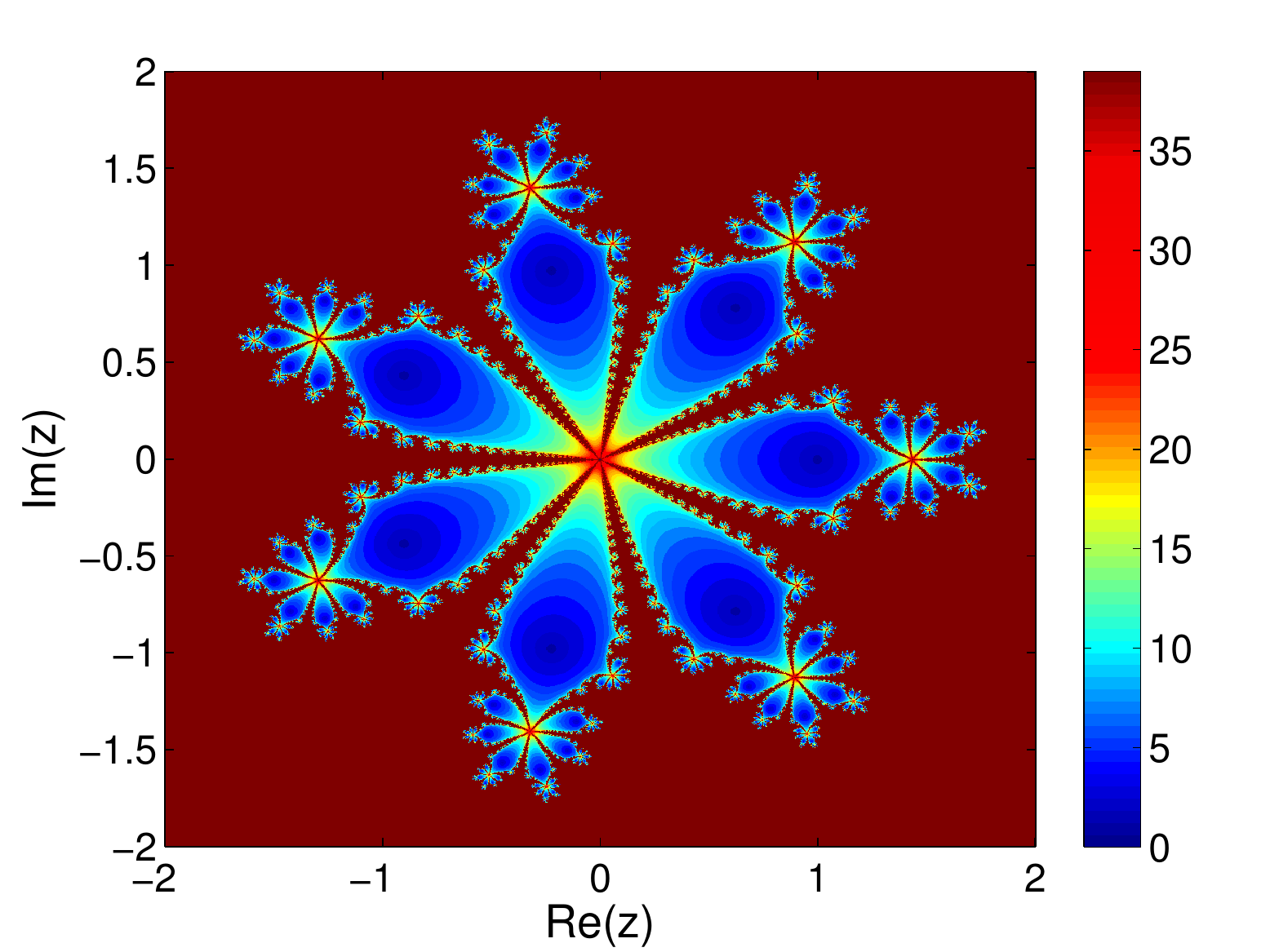}}
\subfigure[$a_1=-1.5$]{\includegraphics[width=0.45\textwidth]{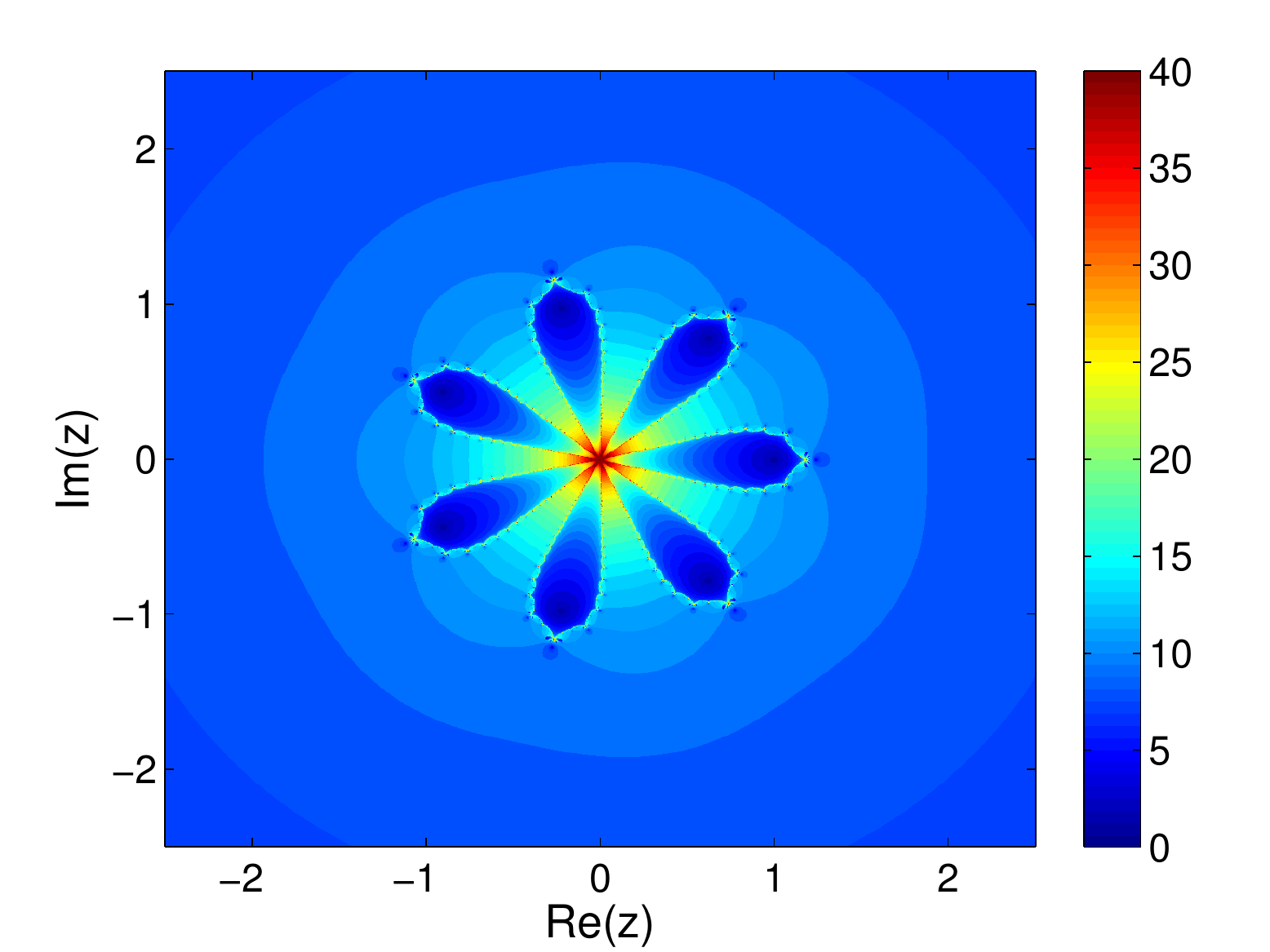}}
\end{center}
\caption{The same as Figure~\ref{fig1} but for some other values of $a_1$ showing interesting convergence behaviours.} \label{fig3}
\end{figure}

Nodules in the complex plane of Figure~\ref{fig1} are represented by the blue coloured regions in Figure~\ref{fig2}. 
Interestingly, one can note that nodules move in space as one varies $a_1$. 
Nonetheless, since the movement is not linear, at one point one can also note that if the value of $a_{1}$ decreases, the size of nodules, in general, will increase. 
From the top right panel of Figure~\ref{fig2}, we can clearly observe that as $a_1$ decreases towards $-1$ the size of nodules becomes larger.

The next point one can note is that when $a_1 > 0$, we basically add an additional zero to the function. 
Yet, from the top right panel of Figure~\ref{fig2}, we can see that there is a range of $a_1 > 0$ for which the method does not converge to $z = 0$, i.e.\
$z = 0$ is a repulsive unstable point. This is represented by the red\footnote{For the interpretation of colour in Figures~\ref{fig1}-\ref{fig3}, the reader is referred to the web version of this article.} coloured region.

Another important point that we need to note here is the presence of a range of $a_1 < 0$ where there is a chaotic region existing in the right half plane. 
It is notable because it implies that by varying $a_1$ we might split a basin of attractions in the complex plane into two basins or unite them into one basin. 
The transition from one case to the other is separated by the presence of a non-convergent region. 
This also means that there is a birth of nodules. 
A clear picture of the nodules birth from a non-convergent region can be seen at the bottom right panel of Figure~\ref{fig2}.

Finally, in Figure~\ref{fig3} we depict the number of iterations needed by our iterative method to converge to a root in the complex plane for several values of $a_1$. 
Regarding the nodules birth, one can compare, e.g., the left panels of Figures~\ref{fig1} and~\ref{fig3}b. 
Notice that the origin in Figure~\ref{fig1} is surrounded by a non-convergent region. 
Yet, in Figure~\ref{fig3}b there is a nodule appear in the middle of the region. 
In this case, we can say that a nodule has been born. 
The expansion of nodules along the boundaries of nodules (see subsequently Figures~\ref{fig1} and~\ref{fig3}c) that eventually makes networks of chaotic regions 
(see Figure~\ref{fig3}d) is another interesting behaviour of the iterative method considered in the present paper.

\section{Conclusions}

In this paper, we have considered a modified Newton method depending on a parameter that can be viewed as an interpolant between several known methods. 
We have utilized the iterative method to find the roots of a function. 
It has been demonstrated that even in such a simple method the convergence behaviour does depend complicatedly on the parameter. 
A point that belongs to the non-convergent region for a particular value of the parameter can be in the convergent region for another parameter value, 
even though the former might have a higher-order of convergence than the second. 
This then indicates that showing whether a method is better than the other should not be done through solving a function from a randomly chosen initial point and comparing the number of iterations needed to converge to a root.

\section*{Acknowledgements}
We would like to thank the anonymous referees for their constructive suggestions that greatly helped to improve the presentation of this paper.





\end{document}